\documentclass[12pt,oneside,reqno]{amsart}

\usepackage{amssymb,latexsym,amsxtra,amscd,ifthen}

\usepackage{amsmath}

\usepackage{amsfonts}

\usepackage{amsthm}

\usepackage{verbatim}

\usepackage{dsfont}

\usepackage{mathtools}

\usepackage{physics}

\usepackage{enumerate}

\usepackage[capitalise]{cleveref}

\usepackage{color}

\usepackage{endnotes}

\usepackage{tikz-cd}

\newtheorem{innercustomgeneric}{\customgenericname}
\providecommand{\customgenericname}{}
\newcommand{\newcustomtheorem}[2]{%
  \newenvironment{#1}[1]
  {%
   \renewcommand\customgenericname{#2}%
   \renewcommand\theinnercustomgeneric{##1}%
   \innercustomgeneric
  }
  {\endinnercustomgeneric}
}

\theoremstyle{plain}

\newtheorem{theorem}{Theorem}[section]

\newtheorem{lemma}[theorem]{Lemma}

\newtheorem{proposition}[theorem]{Proposition}

\newtheorem*{corollary*}{Corollary}

\newtheorem*{theorem*}{Theorem}

\newcustomtheorem{customthm}{Theorem}

\newtheorem*{proposition*}{Proposition}

\theoremstyle{definition}

\newtheorem{example}[theorem]{Example}

\theoremstyle{remark}

\DeclareMathAlphabet{\mathbbold}{U}{bbold}{m}{n}
\def\bb1{\mathbbold{1}}

\def\Fp{\mathbb{F}_p}

\def\bbn{\mathbb{N}}

\DeclareMathOperator\ad{ad}

\DeclareMathOperator\loc{Loc}
\def\Fr{Fr}

\newtheorem*{remark*}{Remark}

\usepackage{xparse}
\DeclareDocumentCommand{\adx}{ O{2} O{x_1}  }{\ad_{#2}^{[#1]}}

\begin{document}
\title{The Restricted Burnside Problem for Moufang Loops}
\author{Alexander Grishkov\protect\endnotemark[1], Liudmila Sabinina\protect\endnotemark[2]$^{,1}$, Efim Zelmanov\protect\endnotemark[3]}
\dedicatory{To the memory of Peter Plaumann.}
	
	%\keywords{growth function, associative algebra, wreath product}
	
	%\subjclass[2010]{Primary: 20E18, 20F50, 20F40, 16R99}


	
\address{1. To whom correspondence should be addressed\hfill\break
E-mail: liudmila@uaem.mx\hfill\break}	
	
\begin{abstract}
We prove that for positive integers $m \geq 1, n \geq 1$ and a prime number $p \neq 2,3$ there are finitely many finite $m$-generated Moufang loops of exponent $p^n$.
\end{abstract}

\maketitle

\section{Introduction}\label{introduction}

A loop $U$ is called a \underline{Moufang loop} if it satisfies the following identities:
\[
((zx)y)x = z ((xy)x) \; \textrm{ and } \; x(y(xz)) = (x(yx))z.
\]
In this paper we solve the Restricted Burnside Problem for Moufang loops of exponent $p^n$, $p>3$.

\begin{customthm}{1}\label{thm:1}
For an arbitrary prime power $p^n$, $p>3$, there exists a function $f(m)$ such that any finite $m$-generated Moufang loop of exponent $p^n$ has order $< f(m)$.
\end{customthm}

For groups this assertion was proved by E. Zelmanov (\cite{20}, \cite{21}).  For Moufang loops of prime exponent it was proved by A. Grishkov \cite{6} (if $p \neq 3$) and G. Nagy \cite{15} (if $p = 3$).  In \cite{16}, \cite{17} the Restricted Burnside Problem was solved for a subclass of Moufang loops and related Bruck loops.

%%%%%%%%%%%%%%%%%%%%%%%%%%%%%%%%%%%%%%%%%%%%%%%%%%%%%%%

\section{Groups with triality}\label{Section2}

A group $G$ with automorphisms $\rho$ and $\sigma$ is called a \underline{group with triality} if $\rho^3 = \sigma^2 = (\rho \sigma)^2 = 1$ and 
\[
[x, \sigma] [x, \sigma]^\rho [x, \sigma]^{\rho^2} = 1
\]
for every $x \in G$, where $[x,\sigma] = x^{-1} x^\sigma$.

Let $G$ be a group with triality.  Let $U = \{ [x,\sigma] | x \in G\}$.  Then the subset $U$ endowed with the multiplication
\[
a \cdot b = (a^{-1})^\rho b (a^{-1})^{\rho^2}; \; a,b \in U
\]
becomes a Moufang loop.

Every Moufang loop $U$ can be obtained in this way from a suitable group with triality, which is finite if $U$ is finite.  Moreover, if $p$ is a prime number, then a finite Moufang $p$-loop can be obtained from a finite $p$-group with triality (\cite{3}, \cite{5}, \cite{10}).

%%%%%%%%%%%%%%%%%%%%%%%%%%%%%%%%%%%%%%%%%%%%%%%%%

\section{Lie and Malcev algebras}\label{Section3}

Let $\Fp$ be a field of order $p$, let $G$ be a group.  Consider the group algebra $\Fp G$ and its fundamental ideal $\omega$, spanned by all elements $1-g, \, g \in G$.  The Zassenhaus filtration is the descending chain of subgroups
\[
G = G_1 > G_2 > \dotsb
\]
where $G_i = \{g \in G | 1-g \in \omega^i\}.$  Then $[G_i, G_j] \subseteq G_{i+j}$ and each factor $G_i / G_{i+1}$ is an elementary abelian $p$-group.  Hence,
\begin{align*}
L = L_p(G) = \sum\limits_{i \geq 1} L_i, \;\;L_i = G_i / G_{i+1}
\end{align*}
is a vector space over $\Fp$.  The bracket
\[
[x_i G_{i+1}, y_j G_{j+1}] = [x_i, y_j] G_{i+j+1}; \; \; x_i \in G_i, y_j \in G_j,
\]
makes $L$ a Lie algebra.  Notice that the bracket $[\; , \; ]$ on the left hand side of the last equality is a Lie bracket whereas $[\; , \; ]$ on the right hand side denotes the group commutator.

Let $x,y$ be generators of a free associative algebra over $\Fp$.  Then $(x+y)^p = x^p + y^p + \{x,y\}$, where $\{x,y\}$ is a Lie element.  Following \cite{12}, we call a Lie $\Fp-$algebra $L$ with an operation $a \to a^{[p]}$, $a \in L$, a Lie $p$-algebra if 
\begin{align*}
(ka)^{[p]} &= k^p a^{[p]},\\
(a+b)^{[p]} &= a^{[p]} + b^{[p]} + \{a,b\},\\
[a^{[p]}, b] &= [\underbrace{a,[a, \dotsc [a}_{p},b] \dotsc]
\end{align*}
for arbitrary $k \in \Fp; a,b \in L$.  The mapping $L_i \to L_{ip}, \linebreak (g_i G_{i+1})^{[p]} = g_i^p G_{ip + 1}$, extends to the operation $a \to a^{[p]}, a \in L$, making $L$ a Lie $p$-algebra.  For more details about this construction see \cite{2}, \cite{11}, \cite{22}.

We call a Lie algebra (resp. Lie $p$-algebra) $L$ with automorphisms $\rho, \sigma$ a \underline{Lie algebra with triality} if $\rho^3 = \sigma^2 = (\rho\sigma)^2 = 1$ and for an arbitrary element $x \in L$ we have
\[
(x^\sigma - x) + (x^\sigma - x)^\rho + (x^\sigma - x)^{\rho^2} = 0.
\]

\begin{lemma}
Let $G$ be a group with triality and let $p$ be a prime number.  Then $L_p(G)$ is a Lie $p$-algebra with triality.
\end{lemma}

\begin{proof}
The automorphisms $\rho, \sigma$ of the group $G$ give rise to automorphisms $\rho, \sigma$ of the Lie algebra $L_p(G)$.  For an element $x_i \in G_i$ we have
\[
[x_i, \sigma][x_i,\sigma]^\rho [x_i, \sigma]^{\rho^2} = 1.
\]
It implies that for the element $x = x_i G_{i+1} \in L_i$ we have
\[
(x^\sigma - x) + (x^\sigma - x)^\rho + (x^\sigma - x)^{\rho^2} = 0.
\]
This completes the proof of the lemma.
\end{proof}

Recall that a (nonassociative) algebra is called a \underline{Malcev algebra} if it satisfies the identities
\begin{enumerate}[(1)]
\item \label{Enum1} $xy = -yx$
\item \label{Enum2} $(xy)(xz) = ((xy)z) x + ((yz)x)x + ((zx)x)y$,
\end{enumerate}
see \cite {4}, \cite{14}, \cite{23}.

\begin{lemma}[see \cite{7}]\label{lem:3.2}
Let $L$ be a Lie algebra with triality over a field of characteristic $\neq 2,3$.  Let $H = \{x \in L | x^\sigma = -x\}$.  Then $H$ is a Malcev algebra with multiplication
\[
a \ast b = [a + 2a^\rho, b] = [a^\alpha, b],
\]
where $a,b \in H$, $\alpha = 1 + 2\rho$.
\end{lemma}

\begin{lemma}\label{lem:3.3}
For arbitrary elements $a,b,c \in H$ we have
\[
3 [[ a,b ], c] = 2(a \ast b) \ast c + (c \ast b) \ast a + (a \ast c) \ast b.
\]
\end{lemma}

We remark that in a Lie algebra with triality over a field $F$, for arbitrary elements $a_1, \dotsc, a_n \in H$ the subspace $\sum\limits_{i=1}^n F a_i + \sum\limits_{i=1}^n F a_i^\alpha = \sum\limits_{i=1}^n F a_i + \sum\limits_{i=1}^n F a_i^\rho$ is invariant with respect to the group of automorphisms $\langle \sigma, \rho \rangle$.

\begin{proof}
Let's prove that for any $x,y,z \in H$:
\begin{equation}\label{eq:xyz}
(x \ast y) \ast z = 2[[x^{\rho^2}, y^\rho], z] + [[x,y],z].
\end{equation}
Using $x+ x^\rho + x^{\rho^2} = 0$ and $y + y^\rho + y^{\rho^2} = 0$, we get
\begin{align*}
v &= [x^\rho,y] - [x,y^\rho] = -[x^{\rho^2}, y] - [x,y] + [x,y^{\rho^2}] + [x,y] = [x,y^{\rho^2}] - [x^{\rho^2}, y];\\
v^\rho &= [x^{\rho^2}, y ^\rho] - [x^\rho, y^{\rho^2}] = - [x^\rho, y^\rho] - [x, y^\rho] + [x^\rho, y^\rho] + [x^\rho, y] = [x^\rho,y] - [x,y^\rho] = v.
\end{align*}
Then
\[
v^\sigma = [x^{\rho \sigma}, y^\sigma] - [x^\sigma, y^{\rho\sigma}] = [x^{\rho^2}, y] - [x, y^{\rho^2}] = -v,
\]
hence $v \in H$ and, by triality, we have $v + v^\rho + v^{\rho^2} = 3v = 0$.  Since the characteristic of the field is not 3 then $v = 0$ and we proved that
\begin{equation}\label{eq:xrhoy}
[x^\rho, y] = [x,y^\rho], \;\; [x^{\rho^2}, y^\rho] = [x^\rho, y^{\rho^2}].
\end{equation}
Finally, we have by \eqref{eq:xrhoy}
\begin{align*}
(x \ast y) \ast z &= [x + 2x^\rho, y] \ast z\\
 &= [[x + 2x^\rho, y], z + 2z^\rho]\\
&= [[x,y],z] + 2[[x^\rho,y],z] + 2[[x+ 2x^\rho,y],z^\rho]\\
&= [[x,y],z] + 2[[x^\rho,y],z] + 2[[x^\rho + 2x^{\rho^2}, y^\rho],z]\\
&= [[x,y],z] + 2[[x^\rho,y],z] + 2[[-x + x^{\rho^2}, y^\rho],z]\\
&= 2[[x^{\rho^2}, y^\rho],z] + [[x,y],z].
\end{align*}
Let $J = J(x,y,z) = (x \ast y) \ast z + (y \ast z) \ast x + (z \ast x) \ast y,$ then by \eqref{eq:xyz} we get
\[
J = 2([[x^{\rho^2}, y^\rho],z] + [[y^{\rho^2}, z^\rho],x] + [[z^{\rho^2}, x^\rho],y]).
\]
But $t = [[x^{\rho^2}, y^\rho],z] - [[z^{\rho^2}, x^\rho],y] = 0$, indeed, we have
\begin{align*}
t - t^\rho &= ([[x^{\rho^2}, y^\rho],z] - [[z^{\rho^2}, x^\rho],y])^\rho - [[x^{\rho^2}, y^\rho],z] + [[z^{\rho^2}, x^\rho], y]\\
&=  [[x,y^{\rho^2}], z^\rho] - [[z,x^{\rho^2}], y^\rho] - [[x^{\rho^2}, y^\rho],z] + [[z^{\rho^2}, x^\rho], y]\\
&=  [[x^{\rho^2}, y], z^\rho] - [[z,x^{\rho^2}], y^\rho] - [[x^{\rho^2}, y^\rho],z] + [[z^{\rho^2}, x^\rho], y]\\
&= [[x^{\rho^2}, z^\rho],y] + [x^{\rho^2},[y,z^\rho]] - [[z,y^\rho],x^{\rho^2}] - [z,[x^{\rho^2}, y^\rho]] \\
& \;\;\;\;- [[x^{\rho^2}, y^\rho],z] + [[z^{\rho^2}, x^\rho], y]\\
&= [x^{\rho^2}, [y,z^\rho]] - [[z,y^{\rho}], x^{\rho^2}]\\
&=0.
\end{align*}
Hence, $t \in H \cap \{v | v^\rho = v\}$.  As above we can prove that $v = 0$, since the characteristic of the field is $\neq$ 3.

Then
\begin{equation}\label{eq:Jxyz}
J(x,y,z) = (x \ast y) \ast z + (y \ast z) \ast x + (z \ast x) \ast y = 6[[x^{\rho^2}, y^\rho],z].
\end{equation}
Now we are ready to prove the Lemma.  By \eqref{eq:xyz} and \eqref{eq:Jxyz} we get
\begin{align*}
& \;2(x \ast y) \ast z + (z \ast y) \ast x + (x \ast z) \ast y \\
= &\; 3(x \ast y) \ast z + (z \ast y) \ast x + (x \ast z) \ast y + (y \ast x) \ast z\\
= & \; 3(x \ast y) \ast z + J(x,y,z)\\
= & \; 6[[x^{\rho^2}, y^\rho], z] + 3[[x,y],z] - 6[[x^{\rho^2}, y^\rho],z]\\
= & \; 3[[x,y],z],
\end{align*}
which proves the lemma.
\end{proof}

\begin{lemma}\label{lem:3.4}
If a Lie algebra $L$ with triality is generated by elements $a_1, \dotsc, a_m, a_1^\alpha, \dotsc, a_m^\alpha$, where $a_1, \dotsc, a_m \in H$, then the Malcev algebra $H$ is generated by $a_1, \dotsc, a_m$.
\end{lemma}

\begin{proof}
We have $L = H \dotplus S$, where $S = \{a \in L | a^\sigma = a\}$ and $H^\alpha \subseteq S$.  Hence the subspace $H$ of $L$ is spanned by left-normed commutators $b = [\dotsc [b_1, b_2], b_3], \dotsc, b_r]$, where $b_1, \dotsc, b_r \in \{a_1, \dotsc, a_m, a_1^\alpha, \dotsc, a_m^\alpha\}$ and elements from $\{a_1, \dotsc, a_m\}$ occur in $b$ an odd number of times.

1. Suppose that $b_r = a_i^\alpha, 1 \leq i \leq m, b' = [ \dotsc [b_1, b_2], \dotsc, b_{r-1}]$.  Then by the induction assumption on $r$ the element $b'$ lies in the Malcev algebra $H'$ generated by $a_1, \dotsc, a_m$ and $b = [b', a_i^\alpha] = - a_i \ast b'$;

2.  Suppose that $b_r \in \{a_1, \dotsc, a_m\}$.  If the element $b_{r-1}$ also lies in $\{a_1, \dotsc, a_m\}$ and $b'' = [ \dotsc [b_1, \dotsc, b_{r-2}]]$ then by the induction assumption $b'' \in H'$.  In this case it remains to use Lemma \ref{lem:3.3}.

Let $b_{r-1} \in \{a_1^\alpha, \dotsc, a_m^\alpha\}$.  Then $b= [\dotsc [b_1, \dotsc, [b_{r-1}, b_r]] + [\dotsc [b_1, \dotsc, b_r], b_{r-1}]$.  By the induction assumption on $r$ applied to the elements $a_1, \dotsc, a_m, \linebreak a_{m+1} = [b_{r-1}, b_r] \in H'$ the first summand lies in $H'$.  The second summand was considered in case 1.  This completes the proof of the lemma.
\end{proof}
%%%%%%%%%%%%%%%%%%%%%%%%%%%%%%%%%%%%%%%%%%%%%%%%%%

\section{Commutator identities in groups}\label{Section4}

Let $\Fr$ be the free group on free generators $x_i, i \geq 1; y, z_1, z_2.$  Recall the Hall commutator identity
\[
[xy,z] = [y,[z,x]] [x,z] [y,z],
\]
where $[x,y] = x^{-1}y^{-1}xy$ is the group commutator.

Let $N$ be the normal subgroup of $\Fr$ generated by the element $y$ and let $N'$ by the subgroup of $N$ generated by $[N,N]$ and by all elements $g^p, g \in N$.  Then $N/N'$ is a vectors space over the finite field $\Fp$.  For an element $g \in \Fr$ consider the linear transformation
\[
g': N/N' \to N/N', \;\; hN' \to [g,h]N'.
\]
Then the Hall identity implies
\[
(ab)' = a' + b' - b'a',
\]
or, equivalently, $1-(ab)' = (1-b')(1-a')$, where $1$ is the identity map.  Hence, $1 - (a^{p^n})' = (1-a')^{p^n}$.  This implies the following well known lemma

\begin{lemma}
$[\underbrace{x_1, [x_1, [ \dotsc [x_1}_{p^n}, y]] \dotsc ] = [x_1^{p^n}, y] \mod N'$.
\end{lemma}

\begin{corollary*}
$[[x_1, z_1], [[x_1, z_1],[\dotsc, [[x_1,z_1],y] \dotsc] = [[x_1, z_1]^{p^n}, y] \mod N'.$
\end{corollary*}

Applying the so called ``collection process'' of G. Higman \cite{11} (see also \cite{22}) we linearize this equality in $x_1$. 

\begin{lemma}\label{lem:4.2}
The product
\[
\prod\limits_{\pi \in S_{p^n}} [[ x_{\pi(1)}, z_1], [[x_{\pi(2)}, z_1], [\dotsc, [[x_{\pi(p^n)}, z_1],y] \dotsc ]
\]
with an arbitrary order of factors lies in the subgroup generated by elements $[[x_{i+1} \dotsm x_{i_r}, z_1]^{p^n}, y]$, $1 \leq i_1 < \dotsb < i_r \leq p^n$, and commutators $c$ in $y,z_1, x_1, \dotsc, x_{p^n}$ such that 
\begin{enumerate}[(1)]
\item $c$ involves all elements $y, x_1, \dotsc, x_{p^n}$,
\item some element $y$ or $x_j, 1 \leq j \leq p^n$, occurs in $c$ at least twice.
\end{enumerate}
\end{lemma}

Consider again a group with triality $G$ and the Lie algebra with triality $L = L_p(G) = \sum\limits_{i=1}^\infty L_i$.  The subspace $H = \{a \in L | a^\sigma + a = 0\}$ is graded, i.e. $H = \sum\limits_{i=1}^\infty H_i$, $H_i = H \cap L_i$.

\begin{lemma}\label{lem:4.3}
Suppose that for an arbitrary element $g \in G$ we have $[g, \sigma]^{p^n} = 1$.  Then
\begin{enumerate}
\item \label{4.3Enum1} for an arbitrary homogeneous element $a \in H_i, i \geq 1$, we have $\ad(a)^{p^n} = 0$,
\item \label{4.3Enum2} for arbitrary homogeneous elements $a_1, \dotsc, a_{p^n}$ from $H$ we have
\[
\sum\limits_{\pi \in S_{p^n}} \ad(a_{\pi(1)}) \dotsb \ad(a_{\pi(p^n)}) = 0.
\]
\end{enumerate}
\end{lemma}

\begin{proof}
For a homogeneous element $a \in H_i$ there exists an element $g\ \in G_i$ such that $a = [g,\sigma]G_{i+1}.$  Then $a^{[p^n]} = [g,\sigma]^{p^n} G_{p^n i + 1} = 0$.  This implies $\ad(a)^{p^n} = \ad(a^{[p^n]}) = 0.$

Let $a_1, \dotsc, a_{p^n}$ be homogeneous elements from $H$, $a_i = [g_i,\sigma]G_{n(i) + 1}, \linebreak g_i \in G_{n(i)}, b = g' G_{j+1}, g'  \in G_j$.  Applying Lemma \ref{lem:4.2} to $x_i = g_i, \linebreak z_1 = \sigma, y = g'$ we get the assertion $(2)$.
\end{proof}

\begin{lemma}\label{lem:4.4}
For an arbitrary element $a \in H$ we have $[a, a^\rho] = 0.$
\end{lemma}

\begin{proof}
We have already mentioned that for an arbitrary element $g \in [G, \sigma]$  we have $[g, g^\rho] = 1$, see \cite{8}.  Hence,  $[g, g^\rho] = 0$ in $L(G)$.

Let $a_i \in H_i, a_j \in H_j$ be homogeneous elements.  We need to show that $[a_i, a_j^\rho] + [a_j, a_i^\rho] = 0$.  There exist elements $g_i \in G_i, g_j \in G_j$ such that $a_i = [g_i, \sigma] G_{i+1}, a_j = [g_j, \sigma]G_{j+1}$.  In the free group $\Fr$ consider the element 
\[
X = [[x_1, z_1], [x_2, z_1]^{z_2}] [[x_2, z_1], [x_1, z_1]^{z_2}].
\]
Applying the Hall identity and the Collection Process in the free group $\Fr$ we get
\[
[[x_1 x_2, z_1], [x_1 x_2, z_1]^{z_2}] = [[x_1, z_1], [x_1, z_1]^{z_2}][[x_2, z_1],[x_2, z_1]^{z_2}] \cdot X  \cdot c_1 \dotsm c_r,
\]
where $c_1, \dotsc, c_r$ are commutators in $x_1, x_2, z_1, z_2$; each of these commutators involved both elements $x_1, x_2$ and at least one of these elements occurs more than once.

Substitute $x_1 = g_i, x_2 = g_j, z_1 = \sigma, z_2 = \rho.$  Then the equality above in the free group $\Fr$ implies $X \in G_{i+j+1}$.  Hence $[a_i, a_j^\rho] + [a_j, a_i^\rho] = 0$, which completes the proof of the lemma.
\end{proof}

\begin{example}
Let $L$ be a nilpotent 3-dimensional Lie algebra with basis $a,b,c$ and multiplication $[a,b]=c, [a,c] = [b,c] = 0$.  The group $S_3$ acts on $L$ via $a^\sigma = -a, b^\sigma = a + b, c^\sigma = c, a^\rho = b, b^\rho = -a - b, c^\rho = c$.  The straightforward computation shows that $L$ is a Lie algebra with triality and that $[a, a^\rho] = -c \neq 0.$
\end{example}

\begin{lemma}\label{lem:4.5}
\begin{enumerate}
\item \label{4.5Enum1} For an arbitrary element $a \in H$, arbitrary $k \geq 1$, we have $\ad(a^\alpha)^{p^k} = \ad(a)^{p^k} + 2 \rho^{-1} \ad(a)^{p^k} \rho$;
\item \label{4.5Enum2} for arbitrary elements $a_1, \dotsc, a_{p^k} \in H$ we have
\begin{align*}
&\sum\limits_{\pi \in S_{p^k}} \ad(a_{\pi(1)}^\alpha) \dotsm \ad(a_{\pi(p^k)}^\alpha) \\
= &\sum\limits_{\pi \in S_{p^k}} \ad(a_{\pi(1)}) \dotsm \ad (a_{\pi(p^k)}) + 2 \rho^{-1} \sum\limits_{\pi \in S_{p^k}} \ad(a_{\pi(1)}) \dotsm \ad (a_{\pi(p^k)}) \rho.
\end{align*}
\end{enumerate}
\end{lemma}

\begin{proof}
We only need to prove part \eqref{4.5Enum1}.  Part \eqref{4.5Enum2} is obtained from \eqref{4.5Enum1} by linearization.  We have $a^\alpha = a + 2a^\rho$.  By Lemma \ref{lem:4.4} $[a, a^\rho] = 0$.  Hence,
\[
\ad(a^\alpha)^{p^k} = \ad(a)^{p^k} + 2^{p^k}\ad(a^\rho)^{p^k} = \ad(a)^{p^k} + 2 \rho^{-1} \ad(a)^{p^k} \rho.
\]
This completes the proof of the lemma.
\end{proof}

We remark that the proof of linearized Engel identity in \cite{6} contains a gap that is filled in this paper.

For an element $a \in H$ let $\ad^*(a)$ denote the operator of multiplication by $a$ in the Malcev algebra, $\ad^*(a): h \to a \ast h, \ad^*(a) = \ad(a^\alpha)$.

\begin{lemma}\label{lem4.6}
\begin{enumerate}
\item \label{4.6Enum1} For an arbitrary homogeneous element $a \in H_i, \linebreak i \geq 1$, we have $\ad^*(a)^{p^n} = 0;$
\item \label{4.6Enum2} for arbitrary elements $a_1, \dotsc, a_{p^n} \in H$ we have 
\[ 
\sum\limits_{\pi \in S_{p^n}}\ad^*(a_{\pi(1)}) \dotsm \ad^*(a_{\pi(p^n)}) = 0.
\] 
\end{enumerate}
\end{lemma}

\begin{proof}
Assertion \eqref{4.6Enum1} follows from Lemma \ref{lem:4.3}.\eqref{4.3Enum1} and Lemma \ref{lem:4.5}.\eqref{4.5Enum1}. Assertion \eqref{4.6Enum2} follows from Lemma \ref{lem:4.3}.\eqref{4.3Enum2} and Lemma \ref{lem:4.5}.\eqref{4.5Enum2}.
\end{proof}

%%%%%%%%%%%%%%%%%%%%%%%%%%%%%%%%%%%%%%%%%%%%%%%%%%%%

\section{Local nilpotence in Malcev algebras}\label{Section5}

\begin{proposition}\label{prop:5.1}
Let $M = M_1 + M_2 + \dotsb$ be a finitely generated graded Malcev algebra over a field of characteristic $p \neq 2,3$, such that
\begin{enumerate}[(i)]
\item \label{PropEnumi} $\ad^*(a)^{p^n} = 0$ for an arbitrary homogeneous element $a \in M$,
\item \label{PropEnumii} $\sum\limits_{\pi \in S_{p^n}} \ad^*(a_{\pi(1)}) \dotsm \ad^*(a_{\pi(p^n)}) = 0$ for arbitrary $a_1, \dotsc, a_{p^n} \in M$.
\end{enumerate}
Then the Malcev algebra $M$ is nilpotent and finite dimensional.
\end{proposition}

If $I$ is an ideal of a Malcev algebra $M$ then $\widetilde {I} = I^2 + I^2 \cdot M$ is also an ideal of $M$.  Consider the descending chain of ideals \linebreak $M^{[0]} = M, M^{[i+1]} = \widetilde{M^{[i]}}$.  We say that a Malcev algebra $M$ is \underline{solvable} if $M^{[n]} = (0)$ for some $n \geq 1$.

\begin{lemma}[V.T. Filippov, \cite{4}]\label{lem:5.1}
A finitely generated solvable Malcev algebra over a field of characteristic $>3$ is nilpotent if and only if each of its Lie homomorphic images is nilpotent.
\end{lemma}

Consider the free Malcev algebra $M(m)$ on $m$ free generators $x_1, \dotsc, x_m$.  As always $\bbn = \{1,2, \dotsc \}$ is the set of positive integers.  The algebra $M(m)$ is $\bbn^m$-graded via 
\[
\deg(x_i) = (0,0, \dotsc, \underset{i}{1}, 0, \dotsc, 0), \;\; 1 \leq i \leq m, \;\; M(m) = \bigoplus\limits_{\gamma \in \bbn^m} M(m)_\gamma.
\]
Let $I$ be the ideal of $M(m)$ generated by elements $\underbrace{a (a ( \dotsb a}_{p^n} b) \dotsb )$ and elements $\sum\limits_{\pi \in S_{p^n}} a_{\pi(1)} (a_{\pi(2)} ( \dotsb (a_{\pi(p^n)} b) \dotsb )$, where $a, a_1, \dotsc, a_{p^n}, b$ run over all homogeneous elements of $M(m)$.  Let
\[
M(m,p^n) = M(m)/I
\]

\begin{lemma}\label{lem:5.2}
The algebra $M(m,p^n)^2$ is finitely generated.
\end{lemma}

\begin{proof}
E.N. Kuzmin (see \cite{14}) showed that for an arbitrary Malcev algebra $M$ we have $M^{[3]} \subseteq M^2 \cdot M^2$.  By \cite{20} every Lie homomorphic image of $M(m,p^n)$ is a nilpotent algebra.  Hence by Lemma \ref{lem:5.1} of V.T. Filippov there exists $t \geq 1$ such that $M(m,p^n)^t \subseteq M(m,p^n)^{[3]} \subseteq M(m,p^n)^2 M(m,p^n)^2$.  Since the algebra $M(m,p^n)$ is $\bbn^m$-graded it implies that $M(m,p^n)^2$ is generated by products of $x_1, \dotsc, x_m$ of length $\ell , 2 \leq \ell \leq t-1$.  This completes the proof of the lemma.
\end{proof}

Recall that an algebra is said to be \underline{locally nilpotent} if every finitely generated subalgebra is nilpotent.

\begin{lemma}\label{lem:5.3}
Let $M = M_1 + M_2 + \dotsb $ be a graded Malcev algebra that satisfies the assumptions \eqref{PropEnumi}, \eqref{PropEnumii} of the Proposition.  Let $I$ be an ideal of $M$ such that both $I$ and $M/I$ are locally nilpotent.  Then the algebra $M$ is locally nilpotent.
\end{lemma}

\begin{proof}
Let $M'$ be a subalgebra of $M$ generated by $m$ homogeneous elements.  Then $M'$ is a homomorphic image of the Malcev algebra $M(m,p^n)$.  By Lemma \ref{lem:5.2} the algebra $(M')^2$ is finitely generated.  Arguing by induction on $\dim (M' + I/I)$ we can assume that the algebra $(M')^2$ is nilpotent.  Hence $M'$ is a solvable algebra.  By \cite{20} all Lie homomorphic images of the algebra $M'$ are nilpotent.  Hence by Lemma \ref{lem:5.1} the algebra $M'$ is nilpotent, which completes the proof of the lemma.
\end{proof}

\begin{lemma}\label{lem:5.4}
Let $M = M_1 + M_2 + \dotsb $ be a graded Malcev algebra that satisfies the assumptions \eqref{PropEnumi}, \eqref{PropEnumii} of the Proposition.  Then $M$ contains a largest graded locally nilpotent ideal $\loc(M)$ such that the factor algebra $M/\loc(M)$ does not contain nonzero locally nilpotent ideals.  
\end{lemma}

\begin{remark*}
For Lie algebras this assertion was proved in \cite{13}, \cite{18}
\end{remark*}

\begin{proof}
Let $I_1, I_2$ be graded locally nilpotent ideals of $M$.  Since the factor algebra $I_1 + I_2/I_1 \cong I_2/ I_1 \cap I_2$ is locally nilpotent it follows from Lemma \ref{lem:5.3} that the algebra $I_1 + I_2$ is locally nilpotent.

Let $\loc(M)$ be the sum of all graded locally nilpotent ideals of $M$.  We showed that the ideal $\loc(M)$ is locally nilpotent.  By Lemma \ref{lem:5.3} the factor algebra $\overline{M} = M/\loc(M)$ does not contain nonzero graded locally nilpotent ideals.  Let $J$ be a nonzero (not necessarily graded) locally nilpotent ideal of $\overline{M}$.  Let $J_{\textrm{gr}}$ be the ideal of $\overline{M}$ generated by nonzero homogeneous components of elements of $J$ of maximal degree.  It is easy to see that the ideal $J_{\textrm{gr}}$ of $\overline{M}$ is locally nilpotent, a contradiction.  This completes the proof of the lemma.
\end{proof}

Recall that an algebra $A$ is called \underline{prime} if for any nonzero ideals $I,J$ of $A$ we have $IJ \neq (0)$.  A graded algebra $A = A_1 + A_2 + \dotsb$ is \underline{graded prime} if for any nonzero graded ideals $I,J$ we have $IJ \neq (0)$.  Passing to ideals $I_{\textrm{gr}}, J_{\textrm{gr}}$ we see that a graded prime algebra is prime.

The proof of the following lemma follows a well known scheme (see \cite{21}).  We still include it for the sake of completeness.

\begin{lemma}\label{lem:5.5}
Let $M = M_1 + M_2 + \dotsb$ be a graded Malcev algebra satisfying the assumptions (i), (ii) of the Proposition.  Then the ideal $\loc(M)$ is an intersection $\loc(M) = \bigcap P$ of graded ideals $P \lhd M$ such that the factor algebra $M/P$ is prime.
\end{lemma}

\begin{proof}
Choose a homogeneous element $a \in M \setminus \loc(M)$.  Since the ideal $I(a)$ generated by the element $a$ in $M$ is not locally nilpotent there exists a finitely generated graded subalgebra $B \subseteq I(a)$ that is not nilpotent.  Since the algebra $B$ satisfies the assumptions (i), (ii) it follows from V.T. Filippov's Lemma \ref{lem:5.1} that the algebra $B$ is not solvable.

Consider the descending chain of subalgebras $B^{(0)} = B, B^{(i+1)} = (B^{(i)})^2$.  Since the algebra $B$ is not solvable we conclude that $B^{(i)} \neq (0)$ for all $i \geq 0$.

By Zorn's Lemma there exists a maximal graded ideal $P$ of $M$ with the property that $B^{(i)} \nsubseteq P$ for all $i$.  Indeed, let $P_1 \subseteq P_2 \subseteq \dotsb$ be an ascending chain of graded ideals such that $B$ is not solvable modulo each of them.  If $B$ is solvable modulo $\bigcup\limits_{i \geq 1} P_i$ then $B^{(s)} \subseteq \bigcup\limits_{i \geq 1} P_i$ for some $s \geq 1$.  By Lemma \ref{lem:5.2} the subalgebra $B^{(s)}$ is finitely generated, hence $B^{(s)} \subseteq P_i$ for some $i$, a contradiction.

We claim that the factor algebra $M/P$ is graded prime.  Indeed, suppose that $I,J$ are graded ideals of $M, P \subsetneq I, P \subsetneq J$, and $IJ \subseteq P$.  By maximality of $P$ there exists $i \geq 1$ such that $B^{(i)} \subseteq I$ and $B^{(i)} \subseteq J$.  Then $B^{(i+1)} \subseteq P$, a contradiction.  This completes the proof of the lemma.
\end{proof}

\begin{proof}[Proof of Proposition 5.1]
Let $M$ be a graded Malcev algebra satisfying the assumptions  \eqref{PropEnumi}, \eqref{PropEnumii}.  If $M$ is not nilpotent then $M \neq \loc(M)$.  By Lemma \ref{lem:5.5}, $M$ has a nonzero prime homomorphic image.  V.T. Filippov \cite{4} showed that every prime non-Lie Malcev algebra over a field of characteristic $p > 3$ is 7-dimensional over its centroid.  Now it remains to refer to the result of E.L. Stitzinger \cite{19} on Engel's Theorem in the form of Jacobson for Malcev algebras.  This completes the proof of the Proposition.
\end{proof}

%%%%%%%%%%%%%%%%%%%%%%%%%%%%%%%%%%%%%%%%%%%%%%%%
\section{Proof of Theorem \ref{thm:1}}\label{Section6}

%Let $U(m,p^n)$ be the free Moufang loop of exponent $p^n$ on $m$ free generators $x_1, \dotsc, x_m$.  Let $D = D(U(m,p^n))$ be the universal group with triality that corresponds to the loop $U(m,p^n), U(m,p^n) = [D, \sigma]$, $D$ is generated by $x_1, \dotsc, x_m, x_1^\rho, \dotsc, x_m^\rho$.  Consider the Zassenhaus descending chain of subgroups $D=D_1 > D_2 > \dotsb $.  Let $G = D/ \bigcap\limits_{i \geq 1} D_i,\linebreak U = [G,\sigma]$.  An arbitrary finite $m$-generated Moufang loop of exponent $p^n$ is a homomorphic image of the loop $U$.  We will show that the loop $U$ is finite.

Let $U(m,p^n)$ be the free Moufang loop of exponent $p^n$ on $m$ free generators $x_1, \dotsc, x_m$.  Let $E = E(U(m,p^n))$ be the minimal group with triality that corresponds to the loop $U(m,p^n)$ (see \cite{9}).  The group $E$ is generated by elements $x_1, \dotsc, x_m, x_1^\rho, \dotsc, x_m^\rho$.  Consider the Zassenhaus descending chain of subgroups $E = E_1 > E_2 > \dotsb.$  Let 
\[
G = E/ \bigcap\limits_{i \geq 1} E_i, \;\;\; U = [G, \sigma].
\]

Theorem 4 from \cite{5} implies that an arbitrary finite $m$-generated Moufang loop of exponent $p^n$ is a homomorphic image of the loop $U$.  We will show that the loop $U$ is finite.

As above, consider the Lie $p$-algebra $L = L_p(G) = \bigoplus\limits_{i \geq 1} L_i$, $L_i=G_i/ G_{i+1}$, over the field $\Fp, \abs{\Fp} = p$, and the Malcev algebra \linebreak$H = \{a - a^{\sigma}| a \in L\}$.  The Malcev algebra $H$ is graded, $H = \bigoplus\limits_{i \geq 1} H_i, H_i = H \cap L_i$, and satisfies the assumptions \eqref{PropEnumi}, \eqref{PropEnumii} of the Proposition.

%Consider the Lie subalgebra $L'$ of $L$ generated by $a_1, \dotsc, a_m$.  Then $H \cap L' = \sum\limits_{i \geq 0} L'_{2i+1}$.  By Lemma \ref{lem:3.4}, $H \cap L'$ is a finitely generated Malcev algebra.  By the Proposition this Malcev algebra is nilpotent and finite dimensional.  Hence the algebra $L'$ is nilpotent and finite dimensional as well.  The Lie algebra $L$ is generated by the elements $a_1, \dotsc, a_m$ as a $p$-algebra.  This means that it is spanned by elements $\rho^{[p^k]}$, where $\rho$ is a commutator in $a_1, \dotsc, a_m; k \geq0$.  The subspace $L'$ is spanned by $\rho^{[p^k]}$, where $\rho$ is a commutator in $a_1, \dotsc,a_m$ of odd length.

Consider the Lie subalgebra $L'$ of $L$ generated by the set $I_m = \{a_1, \dotsc, a_m, a_1^\alpha, \dotsc, a_m^\alpha\}$, where $a_i = x_i E_2 \in L_1, 1 \leq i \leq m$. The whole Lie algebra $L$ is generated by $I_m$ as a $p$-algebra.

Since the subalgebra $L'$ is $S_3$-invariant it follows that $L'$ is a Lie algebra with triality.  Therefore $L'$ gives rise to the Malcev algebra $H' = L' \cap H$.  By Lemma \ref{lem:3.4} the elements $a_1, \dotsc, a_m$ generate $H'$ as a Malcev algebra.  Hence, by Proposition \ref{prop:5.1} the algebra $H'$ is nilpotent and finite dimensional.  Let $\dim_{\Fp}H' = d$.

Since the Lie algebra $L$ is generated by $a_1, \dotsc, a_m$ as a $p$-algebra it follows that $L$ is spanned by $p$-powers $c^{[p^k]},$ where $c$ is a commutator in $a_1, \dotsc, a_m$ of length $\leq 2d, k \geq 0$.  The space $H$ is spanned by $p$-th powers $c^{[p^k]}$, where the commutators $c$ have odd length.

An arbitrary homogeneous element $a \in H_i$ can be represented as $[g,\sigma] G_{i+1}$, where $g \in G_i$.  Hence $[g,\sigma]^{p^n} = 1$ implies $a^{[p^n]} = 0$.  %Let $(L')^d = (0)$.  Then $H$ is spanned by elements $\rho^{[p^k]}$, where $\rho$ is a commutator in $a_1, \dotsc, a_m$ of odd length $< d$ and $k < n$.  Hence $\dim_{\Fp} H < \infty$.  This implies that for a sufficiently large $i$ the restriction of the automorphism $\sigma$ to $L_i$ is identical.
Then $H$ is spanned by $p$-powers $c^{[p^k]}$, where $c$ is a commutator in $a_1, \dotsc, a_m$ of odd length $\leq 2d$ and $k < n$.  Hence, $\dim_{\Fp} H < \infty$.  Since $\abs{H} = \abs{U}$ we conclude that $\abs{U} < \infty$.  This concludes the proof of Theorem 1.

%We claim that the restriction of the automorphism $\sigma$ to the subgroup $G_i$ is identical.  Indeed, let $g \in G_i$ and $[\sigma, g] \in G_k \setminus G_{k+1}, k \geq i$.  Then by Hall's identity we get
%\[
%1 = [\sigma^2, g] = [\sigma, [\sigma,g]] [\sigma, g]^2.
%\]
%Since the vector space $G_k/G_{k+1}$ does not have 2-torsion we conclude that $[\sigma, [\sigma, g]]G_{k+1}$ is a nonzero element from $H \cap L_k$, a contradiction.

%Let $g_1, \dotsc, g_t$ be representatives of right cosets of the sugbroup $G_i$ in $G$.  Again by Hall's identity for an arbitrary element $g \in G_i$ and an arbitrary $1 \leq j \leq t$, we have
%\[
%[g g_j, \sigma] = [g_j, [g,\sigma]][g_j,\sigma][g, \sigma] = [g_j, \sigma].
%\]
%Hence $\abs{[G,\sigma]} \leq t < \infty$.  This completes the proof of the theorem.

\section*{Acknowledgements}
The first author was supported by CNPq (grant 307824/2016-0), by FAPESP (Brazil) and Russian Foundation for Basic Research (under grant 16-01-00577a).  The second author was supported in part by SNI-CONACyT, PRODEP-UAEM "Estancias cortas de investigacion de integrantes de Cuerpos Academicos Consolidados-2020",
FAPESP 2019/24418-0 and UC MEXUS grants.  The third author gratefully acknowledges the support of the NSF grant DMS 1601920.

%\nocite{*}
%\bibliographystyle{amsplain}
%\bibliography{MoufangLoopsbib}

\endnotetext[1]{Department of Mathematics, University of S\~ao Paulo, S\~ao Paulo, Brazil, \\
and Omsk F.M. Dostoevsky State University, Omsk, Russia,\\
E-mail address: grishkov@ime.usp.br;}

\endnotetext[2]{Department of Mathematics, Autonomous University of the State of Morelos, Cuernavaca, Mexico,\\
E-mail address: liudmila@uaem.mx;}

\endnotetext[3]{Department of Mathematics, University of California, San Diego, USA\\
E-mail address: ezelmanov@math.ucsd.edu}

\theendnotes

\end{document}